\documentclass[leqno,10pt]{amsart}

\usepackage{amssymb,amsthm}
\usepackage{enumerate}
\usepackage{graphics,color}

\usepackage{amsmath}

\newtheorem{theorem}{Theorem}[section]
\newtheorem{corollary}[theorem]{Corollary}

\newtheorem{proposition}[theorem]{Proposition}
\theoremstyle{definition}
\newtheorem{definition}[theorem]{Definition}
\newtheorem{assumptions}[theorem]{Assumption}
\newtheorem{example}[theorem]{Example}

\usepackage[latin2]{inputenc}
\usepackage[T1]{fontenc}
\usepackage[magyar,english]{babel}
\usepackage{ae,aecompl}

\def\ee{\mathrm{e}}
\def\dd{\mathrm{d}}

\def\dt{\tfrac{\dd}{\dd t}}
\def\C{\mathrm{C}}

\def\CC{{\mathbb{C}}}
\def\RR{{\mathbb{R}}}
\def\NN{{\mathbb{N}}}

\def\calL{{\mathcal{L}}}

\renewcommand{\r}{\right}
\renewcommand{\l}{\left}

\begin{document}

\title[Convergence of Magnus method]{The norm convergence of a Magnus expansion method}

\author[A. B\'{a}tkai]{Andr\'{a}s B\'{a}tkai}
\address{Institute of Mathematics, E\"{o}tv\"{o}s Lor\'{a}nd University,   P\'{a}zm\'{a}ny P. s\'{e}t\'{a}ny 1/C, Budapest, 1117, Hungary.}
\email{batka@cs.elte.hu}
\thanks{Supported by the Alexander von Humboldt-Stiftung}

\author[E. Sikolya]{Eszter Sikolya}
\email{seszter@cs.elte.hu}


\date\today
\begin{abstract}
We consider numerical approximation to the solution of non-autonomous evolution equations. The order of convergence of the simplest possible Magnus method will be investigated.
\end{abstract}
\maketitle
%

\section{Introduction}

The theoretical analysis of numerical methods and product formulae for partial differential equations plays an increasingly important role also in functional analysis.
In this paper we are interested in non-autonomous evolution equations of type
\begin{equation*}
\begin{cases}\tag{$\mathrm{NCP}_{s,x}$}
\dt u(t)=A(t)u(t), \quad t\geq s\in\RR, \\
u(s)=x\in X,
\end{cases}
\end{equation*}
where $X$ is a Banach space, $\big(A(t), D(A(t))\big)$ is a family of (usually unbounded) linear operators on $X$. Such kind of problems arise in many applications in quantum physics, Hamiltonian dynamics, transport problems, etc. Among the numerical methods for the approximation of the solution, the Magnus integrators play an important role, see Isereles et al. \cite{iserles-marthinsen} or Magnus \cite{magnus}. The basic idea of this method is to express the solution $u(t)$ in the form
\begin{equation*}
u(t)=\exp(\Omega(t))x,
\end{equation*}
where $\Omega(t)$ is an infinite sum yielded by the formal Picard iteration
\begin{align}
\Omega(t)=\int_0^tA(\tau)\, \dd\tau-&\frac{1}{2}\int_0^t\left[\int_0^{\tau}A(\sigma)\,\dd\sigma ,A(\tau)\right]\, \dd\tau\notag\\
+&\frac{1}{4}\int_0^t\left[\int_0^{\tau}\left[\int_0^{\sigma}A(\mu)\, \dd\mu, A(\sigma)\right]\,\dd\sigma ,A(\tau)\right]\, \dd\tau\label{omegat}\\
+&\frac{1}{12}\int_0^t\left[\int_0^{\tau}A(\sigma)\,\dd\sigma, \left[\int_0^{\tau}A(\mu)\, \dd\mu, A(\tau)\right]\right]\, \dd\tau+\cdots\notag
\end{align}
(with $s=0$), where $[U,V]=UV-VU$ denotes the commutator of the operators $U$ and $V$. Numerical methods based on this expansion are presented in \cite{iserles-munthe} by Iserles et al. or in Hochbruck and Ostermann \cite{hochbruck-ostermann}. They are of the form
\begin{equation}\label{magnus_method}
y_{n+1}=\exp(\Omega_n) y_n,
\end{equation}
to give an approximation of $y(t_{n+1})$ at $t_{n+1}=t_n+h$. Here $\Omega_n$ is a suitable approximation of $\Omega(h)$ given by \eqref{omegat} when $A(\tau)$ is substituted by $A(t_n+\tau)$.

We want to investigate the convergence of the Magnus method where the midpoint rule is used, that is, in \eqref{magnus_method}
\begin{equation*}
\Omega_n=hA\left(t_n+\frac{h}{2}\right).
\end{equation*}

In Section 2 we give a short overview on the theory of non-autonomous evolution equations. In Section 3 the norm convergence of the above Magnus method is proven, with an application to the Schr\"{o}dinger equation. Finally, in Section 4 we explore a slight generalizaton with applications to hyperbolic equations.

\section{Non-autonomous evolution equations}

In this section we summarize the main results and definitions on
\textit{non-autonomous} evolution equations and evolution semigroups needed
for our later exposition. For a detailed account and bibliographic
references see, e.g., the survey by Schnaubelt in \cite[Section
VI.9.]{Engel-Nagel}. Consider now the non-autonomous Cuachy problem
\begin{equation*}
\begin{cases}\tag{$\mathrm{NCP}_{s,x}$}
\dt u(t)=A(t)u(t), \quad t\geq s\in\RR, \\
u(s)=x\in X,
\end{cases}
\end{equation*}
where $X$ is a Banach space, $\big(A(t), D(A(t))\big)$ is a
family of (usually unbounded) linear operators on $X$.

\begin{definition} A continuous function $u : [s,\infty) \rightarrow X$ is called a \emph{(classical) solution} of ($\mathrm{NCP}_{s,x}$) if $u \in \C^1([s, \infty);X)$, $u(t) \in D(A(t))$ for all $t \ge s$, $u(s) = x$, and $\dt{u}(t) = A(t) u(t)$ for $t \ge s$.
\end{definition}

We use the following standard definition for the well-posedness of the non-autonomous Cauchy problem associated to a given family of linear operators.
\begin{definition} \label{Well-NCP}
For a family $\big(A(t),D(A(t))\big)_{t \in \RR}$ of linear operators on the Banach space $X$ \emph{the non-autonomous Cauchy problem is well-posed} (with regularity subspaces $(Y_s)_{s \in \RR}$ and exponentially bounded solutions)
if the following are true.
\begin{enumerate}[(i)]
\item \emph{(Existence)} For all $s \in \RR$ the subspace
\begin{equation*} Y_s := \Bigl\{ y \in X \; : \; \mbox{ there exists a classical solution for } (\mathrm{NCP}_{s,y})\Bigr\} \subset D(A(s))
\end{equation*}
is dense in $X$.
\item \emph{(Uniqueness)} For every $y \in Y_s$ the solution $u_s(\cdot,y)$ is unique.
\item \emph{(Continuous dependence)} The solution depends continuously on $s$ and $y$, i.e.,  if $s_n \to s \in \RR$, $y_n \to y \in Y_s$ with $y_n \in Y_{s_n}$ but in the topology of $X$, then we have
\begin{equation*}
\| \hat{u}_{s_n}(t,y_n) - \hat{u}_s(t,y) \| \to 0
\end{equation*}
uniformly for $t$ in compact subsets of $\RR$, where
\begin{equation*}
\hat{u}_r(t,y) :=
\begin{cases}
 u_r(t,y) &\mbox{if } r \leq t, \\
y &\mbox{if } r > t.
\end{cases}
\end{equation*}
\item \emph{(Exponential boundedness)} There exist constants $M \ge 1$ and $\omega \in \RR$
such that
\begin{equation*}
\| u_s(t, y)\| \leq M \ee^{\omega (t-s)} \| y \|
\end{equation*}
for all $y \in Y_s$ and $t\ge s$.
\end{enumerate}
\end{definition}

\noindent As in the autonomous case, the operator family solving a non-autonomous Cauchy problem enjoys certain algebraic properties.

\begin{definition}\label{evf}
A family $U=(U(t,s))_{t \ge s}$ of linear, bounded operators on a Banach space $X$ is called an (exponentially bounded)
\emph{evolution family} if
\begin{enumerate}
\item[(i)] $U(t,r)U(r,s) = U(t,s), \quad U(t,t) = I$ \quad holds for all $t \ge r \ge s \in \RR$,
\item[(ii)] the mapping $(t,s) \mapsto U(t,s)$ is strongly continuous,
\item[(iii)] $\| U(t,s) \| \leq M \ee^{\omega (t-s)}$ for some $M \ge 1, \omega \in \RR$ and all $t \ge s \in \RR$.
\end{enumerate}
\end{definition}

In general, however, and in contrast to the behavior of $C_0$-semigroups (i.e., the autonomous case), the algebraic
properties of an evolution family do not imply any differentiability on a dense subspace. So we have to add some differentiability
assumptions in order to solve a non-autonomous Cauchy problem (abbreviated later on as NCP) by an evolution family.
\begin{definition}\label{Well-evf}
An evolution family $U=(U(t,s))_{t \ge s}$ is called \emph{evolution family solving NCP} if for every $s \in \RR$ the
regularity subspace
\begin{equation*}
Y_s := \Bigl\{ y \in X \; : \; [s, \infty) \ni t \mapsto U(t,s)y \mbox{ solves }  (\mathrm{NCP}_{s,y}) \Bigr\}
\end{equation*}
is dense in $X$.
\end{definition}

In this case, the unique classical solution of ($\mathrm{NCP}_{s,x}$) is given by $u(t):=U(t,s)x$.
The well-posedness of NCP can now be characterized by
the existence of a solving evolution family, see \cite[Proposition VI.9.3]{Engel-Nagel}.

\begin{proposition}
Let $X$ be a Banach space, and let $\big(A(t),D(A(t))\big)_{t \in \RR}$ be a family of linear operators on $X$. The following assertions are equivalent.
\begin{enumerate}
\item[(i)]  The associated non-autonomous Cauchy problem is well-posed.
\item[(ii)] There exists a unique evolution family
  $(U(t,s))_{t \ge s}$ solving NCP.
\end{enumerate}
\end{proposition}
We say in this case that the evolution family is \emph{generated} by the operators $\big(A(t),D(A(t))\big)_{t \in \RR}$.

Unfortunately, the well-posedness of non-autonomous evolution equations is a complicated issue and there is no general theory describing it. Conditions implying well-posedness are generally divided into
assumptions of \emph{``parabolic''} and of \emph{``hyperbolic''}
type. Roughly speaking, the main difference between these two types
is that in the parabolic case we assume all $A(t)$ being generators
of analytic semigroups, while in the hyperbolic case we assume the
stability for certain products instead. In both cases one has to add
some continuity assumption on the mapping $t \mapsto A(t)$.
We mention only a typical and quite simple version for each type.

\begin{assumptions}[\emph{Parabolic case}] \label{asu-para} \label{A(inh-para)}
\rule{0pt}{0pt}
\begin{enumerate}[(P1)]
\item \label{P1} The domain $D:= D(A(t))$ is dense in $X$
and is independent of $t \in \RR$.
\item\label{P2} For each $t \in \RR$ the operator $A(t)$ is
the generator of an analytic semigroup $\ee^{\cdot A(t)}$. For all
$t \in \RR$, the resolvent $R(\lambda, A(t))$ exists for all
$\lambda \in \CC$ with $\Re \lambda \ge 0$ and there is a constant
$M \ge1$ such that
\begin{equation*} \| R(\lambda, A(t)) \| \leq \frac{M}{|\lambda| +1}
\end{equation*}
for $\Re \, \lambda \ge 0$, $t \in \RR$. The semigroups $\ee^{\cdot
A(t)}$ satisfy $\|\ee^{s A(t)}\| \leq M\ee^{\omega s}$
for absolute constants $\omega < 0$ and $M \ge 1$.
\item \label{P3} There exist constants $L \ge 0$ and
$0 < \alpha \leq 1$ such that
\begin{equation*}
\| (A(t) - A(s))A(0)^{-1} \| \leq L |t-s|^{\alpha} \mbox{ for all }
t,s \in \RR.
\end{equation*}
\end{enumerate}
\end{assumptions}
\begin{assumptions}[\emph{Hyperbolic case}] \label{asu-hyp}
\rule{0pt}{0pt}
\begin{enumerate}[(H1)]
\item The family $(A(t))_{t \in \RR}$ is \textit{stable},
i.e., all operators $A(t)$ are generators of $C_0$-semi\-groups and
there exist constants $M \ge 1$ and $\omega \in \RR$ such that
\begin{equation*} (\omega, \infty) \subset \rho(A(t)) \quad \mbox{for all } t \in \RR
\end{equation*}
and
\begin{equation*} \label{stab_hyp}
\Bigl\| \prod_{j=1}^{k} R(\lambda, A(t_j)) \Bigr\| \leq M (\lambda
-\omega)^{-k} \quad \mbox{for all } \lambda > \omega
\end{equation*}
and every finite sequence $- \infty < t_1 \leq t_2 \leq \dots \leq
t_k < \infty$, $k\in \NN$.
\item There exists a densely embedded subspace
$Y \hookrightarrow X$, which is a core for every $A(t)$ such that
the family of the parts $(A_{|Y}(t))_{t \in \RR}$ in $Y$ is a stable
family on the space $Y$.
\item The mapping
$\RR \ni t \mapsto A(t) \in {\mathcal L}(Y,X)$ is uniformly
continuous.
\end{enumerate}
\end{assumptions}

Let us close this summary by recalling an important basic perturbation result, see Engel and Nagel \cite[Theorem VI.9.19]{Engel-Nagel}.

\begin{theorem}\label{thm:var_const}
Let $(U(t,s))_{t \ge s}$ be an evolution family in the Banach space $X$. Let $B(t)$, $t\in\RR$ be closed operators such that $U(t,s)X\subset D(B(t))$, $t\mapsto B(t)U(t,s)$ is strongly continuous and $\|B(t)U(t,s)\|\leq k(t-s)$ for $t>s$ and some locally integrable function $k$. Then there is a unique evolution family $(U_B(t,s))_{t \ge s}$ such that
\begin{equation}\label{eq:variationof_const1}
U_B(t,s)x = U(t,s)x + \int_s^t U_B(t,r)B(r)U(r,s)x \dd r
\end{equation}
for all $x\in X$ and $t>s$.
\end{theorem}

Note that this result only states the existence of an evolution family satisfying a variations-of constants formula, but not the well-posedness of non-autononous Cauchy problem associated to the operators $(A(t)+B(t))_{t\in\RR}$.

\section{Norm convergence of the Magnus method}

We would like to investigate the convergence of the following, simplest possible Magnus method introduced in the Introduction, where we take the first term of the Magnus series expansion and evaluate the integral using the midpoint rule. Let us take a time step $h>0$, and consider the following iteration:
\begin{equation}\label{eq:magnus_split}
y_{n+1}=\ee^{\Omega_n} y_n,\qquad \Omega_n = h A\left(t_n+\tfrac{h}{2}\right),
\end{equation}
starting with $t_0=s$, $y_0=x$ and then $t_{n+1}=t_n+h$. Here and later on we will use the exponential notation to denote strongly continuous semigroups.

\begin{assumptions}\label{ass:magnus_norm}
Assume that
\begin{enumerate}[(a)]
\item  there exists $K\geq 1$, $\omega\in\RR$ such that $\|\ee^{tA(r)}\|\leq K\ee^{\omega t}$ for all $r,t\in \RR$. Further, $A(t)=A+V(t)$, where $A$ is the generator of a strongly continuous semigroup and $V(t)\in \calL(X)$ for all $t\in\RR$,
\item (\emph{Well-posedness}) the non-autonomous Cauchy problem is well-posed, i.e., there is an evolution family $U(t,s)$ which solves $(\text{NCP}_{s,x})$ and satisfies $\|U(t,s)\|\leq K\ee^{\omega(t-s)}$.

\item (\emph{Stability}) the Magnus method \eqref{eq:magnus_split} is stable, i.e., there are constants $M\geq 1$, $\omega\in \RR$ such that for all $n\in\NN$
\begin{equation*}
\left\| \prod_{j=0}^k \ee^{\tfrac{t-s}{n}A\left(s+\tfrac{(2j+1)(t-s)}{2n}\right)} \right\| \leq M \ee^{\tfrac{(k+1)\omega(t-s)}{n}}
\end{equation*}
for all $k=0,1\ldots,n-1$, and
\item (\emph{Local H\"{o}lder continuity}) there exists $\alpha\in (0,1]$ and for all $K>0$ there is $L=L(K)>0$ such that
\begin{equation*}
\left\| A(t)-A(s)\right\| \leq L|t-s|^{\alpha}
\end{equation*}
for all $t,s\in (-K,K)$.
\end{enumerate}
\end{assumptions}

Here and later on, for bounded linear operators $L_k\in \calL(X)$,
$$\prod_{k=0}^{n-1}L_k:=L_{n-1} L_{n-2}\cdots L_0$$
denotes the ``time-ordered product''. If the product is empty, we define it as the identity operator $I$ on $X$.

Note that the last condition (d) is always satisfied with $\alpha=1$ if we assume continuous differentiability of the map $t\mapsto V(t)$. Note also that the boundedness condition in (a) is always satisfied if the map $t\mapsto V(t)$ is bounded.

\begin{theorem}\label{thm:magnus conv}
Assume that  the conditions in Assumption \ref{ass:magnus_norm} are satisfied. Then the Magnus method defined by \eqref{eq:magnus_split} converges in the operator norm and has a uniform convergence order $\alpha\in (0,1]$, i.e., for all $t,s\in \RR$, $t>s$ there is $C\geq 1$ such that
\begin{equation*}
\left\| U(t,s)-\prod_{k=0}^{n-1}\ee^{\frac{t-s}{n} A(s+(2k+1)\frac{t-s}{2n})}\right\| \leq \frac{C}{n^\alpha}
\end{equation*}
for all $n\in\NN$.
\end{theorem}

\begin{proof}
First note that  $A(t)=A(s+\tfrac{h}{2})+B(t)$ with $B(t)=V(t)-V(s+\tfrac{h}{2})\in \calL(X)$. Hence, we can apply the variation of constants formula \eqref{eq:variationof_const1} and obtain that
\begin{equation*}
U(s+h,s)x-\ee^{hA(s+h/2)} x = \int_s^{s+h} U(s+h,r)B(r) \ee^{(r-s)A(s+h/2)}x \dd r
\end{equation*}
Here the left-hand side is the difference of the exact solution and the solution obtained by the Magnus method after one time step (the so-called local error).

Clearly, by the assumptions on $V$, we have
\begin{equation*}
\|B(r)\| = \|V(r)-V(s+h/2)\| \leq L|r-s-h/2|^{\alpha}\leq Lh^{\alpha}
\end{equation*}
for $r\in (s,s+h)$. Hence,
\begin{align}\label{eq-Magnusbiz1}
\left\| U(s+h,s)x-\ee^{hA(s+h/2)} x \right\| &\leq \int_s^{s+h} K\ee^{\omega(s+h-r)} L h^{\alpha} K\ee^{\omega(r-s)} \|x\| \dd r \\
&= K^2 L \ee^{\omega h} h^{\alpha+1} \|x\|,\notag
\end{align}
proving the so-called consistency estimate we need for the following step.

The convergence order of the method can be established now using the standard telescopic argument. Let $h:=\frac{t-s}{n}$. Then
\begin{equation*}
U(t,s)=\prod_{k=0}^{n-1}U(s+(k+1)h,s+kh)
\end{equation*}
and the Magnus approximation is
\begin{equation*}
W(t,s)=\prod_{k=0}^{n-1}\ee^{hA(s+(2k+1)\frac{h}{2})}.
\end{equation*}
Hence, for any $x\in X$,
\begin{multline*}
U(t,s)x-W(t,s)x=\\
=\sum_{j=0}^{n-1}\prod_{k=j+1}^{n-1}U(s+(k+1)h,s+kh)\l[U(s+(j+1)h,s+jh)-\ee^{hA(s+(2j+1)\frac{h}{2})}\r]\\
\times \prod_{l=0}^{j-1}\ee^{hA(s+(2l+1)\frac{h}{2})}x.
\end{multline*}
From \eqref{eq-Magnusbiz1} we have
\begin{equation*}
\left\| U(s+(j+1)h,s+jh)x-\ee^{hA(s+(2j+1)\frac{h}{2})}x\right\| \leq  K^2 L \ee^{\omega h} h^{\alpha+1} \|x\|,
\end{equation*}
hence, for $n$ sufficiently large, we obtain that
\begin{align*}
\left\| U(t,s)-W(t,s)\right\| &\leq \sum_{j=0}^{n-1} K\ee^{(n-j-1)\omega h}\cdot K^2 L \ee^{\omega h} h^{\alpha+1}\cdot M\ee^{j\omega h}\\& = K^3 L M h^{\alpha+1} n\ee^{n\omega h} =  K^3LM(t-s)^{\alpha+1}\ee^{\omega(t-s)}\cdot\frac{1}{n^{\alpha}},
\end{align*}
since $n=\frac{t-s}{h}$.
\end{proof}

\begin{example}
Motivated by Hochbruck and Lubich \cite[Section 8.]{hochbruck-lubich}, we consider the $d$-dimensional Schr\"{o}dinger equation
\begin{equation*}
\begin{cases}
i\tfrac{\partial\psi}{\partial t} = -\tfrac{1}{2}\Delta \psi + b(x,t)\psi, \qquad x\in \RR^d,\, t>0,\\
\psi(x,0)=\psi_0(x),
\end{cases}
\end{equation*}
where the potential $b$ and the initial function $\psi_0$ are assumed to be $2\pi$-periodic in every spatial coordinate.

Assume further that
\begin{itemize}
\item the function $b$ is real-valued,
\item $b(\cdot,t)\in L^{\infty}(\RR^d)$ for all $t>0$, and
\item the function $t\mapsto b(\cdot,t)\in L^{\infty}(\RR^d)$ is $\alpha$-H\"{o}lder continuous for some $\alpha\in (0,1]$.
\end{itemize}
Then the operator family $A(t)= \tfrac{i}{2}\Delta -ib$ satisfies the conditions of Assumptions \ref{asu-hyp}, hence the Schr\"{o}dinger equation is well-posed. Clearly, Assumptions \ref{ass:magnus_norm} are also satisfied, hence the Magnus method \eqref{eq:magnus_split} has operator norm convergence of order $\alpha$.

Note that Hochbruck and Lubich showed that under additional smoothness conditions on $b$, one can recover the classical second order convergence for smooth initial values. This convergence will be, however, in general only strong, the uniform rate remains first order.
\end{example}

We can use the above result to prove a regularity result on the propagators. Note that the compactness of propagators is important, for example, in case you want to establish a Peano type existence result on semilinear equations.
\begin{corollary}
Assume that the conditions in Assumption \ref{ass:magnus_norm} are satisfied and that there is $r_0\in \RR$ such that $A(r_0)$ generates an immediately compact semigroup. Then all the propagators $U(t,s)$ are compact.
\end{corollary}

\begin{proof}
We know by the bounded perturbation theorem that all the semigroups $\ee^{tA(r)}$ are immediately compact for $r\in \RR$. Since the Magnus method \eqref{eq:magnus_split} converges in the operator norm, it follows that the operators $U(t,s)$ are compact.
\end{proof}

\section{Generalization to strong convergence}

If we weaken the assumptions, the previous result can be generalized in a straightforward way. We weaken the assumptions so that we only assume $A(t)=A+V(t)$ where $V(t)\in \calL(D(A),X)$ is allowed to be an unbounded operator.

\begin{assumptions}\label{ass:magnus_strong}
Assume that
\begin{enumerate}[(a)]
\item the non-autonomous Cauchy problem is well-posed and solved by the evolution family $U(t,s)$, the operator $A(r)$ is the generator of a strongly continuous semigroup for all $r\in\RR$, $D:=D(A(r))$, all the graph norms are equivalent with the same constants, and there is $K\geq 1$, $\omega\in \RR$ such that
    \begin{equation*}
    \|U(t,s)\| \leq K \ee^{\omega(t-s)}, \quad   \|\ee^{tA(r)}\| \leq K\ee^{\omega t}
    \end{equation*}
    for all $t,r\in\RR$,
\item (\emph{Stability}) the Magnus method \eqref{eq:magnus_split} is stable, i.e., there are constants $M\geq 1$, $\omega\in \RR$ such that for all $n\in\NN$
\begin{equation*}
\left\| \prod_{j=0}^k \ee^{\tfrac{t-s}{n}A\left(s+\tfrac{(2j+1)(t-s)}{2n}\right)} \right\| \leq M \ee^{\tfrac{(k+1)\omega(t-s)}{n}}
\end{equation*}
for all $k=0,1\ldots,n-1$, and
\item (\emph{Local H\"{o}lder continuity}) there exists $\alpha\in (0,1]$ and for all $K>0$ there is $L=L(K)>0$ such that
\begin{equation*}
\left\| A(t)-A(s)\right\|_D \leq L|t-s|^{\alpha}
\end{equation*}
for all $t,s\in (-K,K)$, where $\|\cdot \|_D$ denotes (one of) the graph norm(s) on $D.$
\end{enumerate}
\end{assumptions}

\begin{proposition}
Assume that  the conditions in Assumption \ref{ass:magnus_strong} are satisfied. Then the Magnus method defined by \eqref{eq:magnus_split} converges for all $x\in X$ and has for smooth initial values $x\in D$ convergence order $\alpha\in (0,1]$.
\end{proposition}

The proof goes on the same line as that of Theorem \ref{thm:magnus conv} and is therefore omitted.

\begin{example}
Consider the Schr\"{o}dinger equation
\begin{equation*}
\begin{cases}
i\tfrac{\partial\psi}{\partial t} = -A(t)\psi, \qquad x\in \RR^d,\, t>0,\\
\psi(x,0)=\psi_0(x),
\end{cases}
\end{equation*}
where, formally,
\begin{equation*}
A(t):=\sum_{j,k=1}^{d}\partial_j a_{j,k}(\cdot,t)\partial_k
\end{equation*}
is the differential operator corresponding to the real-valued coefficients $a_{j,k}$. See Engel and Nagel \cite[Section VI.5.c]{Engel-Nagel} for further details on such operators. We assume here that the coefficients $a_{j,k}$ satisfy
\begin{equation*}
a_{j,k}(\cdot, t)=a_{k,j}(\cdot,t)\in W^{1,\infty}(\RR^d),
\end{equation*}
and
\begin{equation*}
\sum_{j,k=1}^d a_{j,k}(x,t)y_j y_k \geq c|y|^2
\end{equation*}
for a uniform constant $c>0$ and for all $x,y\in \RR^d$. Then $D(A(r))=H^2(\RR^d)$ with graph norms equivalent to the Sobolev norms. Differential operators of this type usually appear as the linearization of non-autonomous Schr\"{o}dinger equations.

Assume further that there is $\alpha\in (0,1]$ such that the maps $t\mapsto a_{j,k}(\cdot,t)$ are $\alpha$-H\"{o}lder continuous. Since all the operators $A(r)$ are selfadjoint, $iA(r)$ generates a unitary group, hence the hyperbolicity Assumption \ref{asu-hyp} are satisfied.

Hence, for all $\psi_0\in H^2(\RR^d)$ the Magnus method \eqref{eq:magnus_split} converges and its order is at least $\alpha$.
\end{example}

\section*{Acknowledgments}
A.~B\'atkai was supported by the Alexander von Humboldt-Stiftung and by the OTKA grant Nr. K81403. E.~Sikolya was supported by the Bolyai Grant of the Hungarian Academy of Sciences. Both authors were supported by the European Union and co-financed by the European Social Fund (grant agreement no. TAMOP
4.2.1./B-09/1/KMR-2010-0003).

\end{document}